\documentclass{article}
\usepackage{amsmath}
\usepackage{amsfonts}
\usepackage{amsthm}

\newtheorem{theorem}{Theorem}[section]
\newtheorem{definition}[theorem]{Definition}
\newtheorem{lemma}[theorem]{Lemma}
\newtheorem{remark}[theorem]{Remark}
\newtheorem{corollary}[theorem]{Corollary}
\newtheorem{example}[theorem]{Example}
\newtheorem{proposition}[theorem]{Proposition}
\newtheorem{observation}[theorem]{Observation}

\def\ma{\mathcal{A}}

\def\me{\mathcal{E}}
\def\mh{\mathcal{H}}

\def\ml{\mathcal{L}}

\def\mf{\mathcal{P}}
\def\mr{\mathcal{R}}
\def\ms{\mathcal{S}}
\def\mju{\mathcal{U}}

\def\mw{\mathcal{W}}
\def\mx{\mathcal{X}}
\def\my{\mathcal{Y}}
\def\mz{\mathcal{Z}}
\def\we{\mathcal{WE}}

\newcommand{\codim}{{\rm codim}\hskip0.02cm}

\newcommand{\dist}{{\rm dist}\hskip0.02cm}

\newcommand{\lin}{{\rm lin}\hskip0.02cm}
\newcommand{\im}{{\rm im}\hskip0.02cm}

\newcommand{\ep}{\varepsilon}
\newcommand{\tr}{{\rm tr}\hskip0.02cm}
\newcommand{\WOT}{{\rm WOT}}
\newcommand{\SOT}{{\rm SOT}}

\begin{document}
\title{\LARGE{ Weak operator topology,
operator ranges and operator equations via Kolmogorov widths}}

\author{M.\,I.~Ostrovskii\\
Department of Mathematics and Computer Science\\
St. John's University\\
8000 Utopia Parkway\\
Queens, NY 11439\\
USA\\
e-mail: {\tt ostrovsm@stjohns.edu} \and\\
V.\,S.~Shulman\\
Department of Mathematics\\
Vologda State Technical University\\
15 Lenina street\\
Vologda 160000\\
RUSSIA\\
e-mail: {\tt shulman\_v@yahoo.com}}

\date{\today}
\maketitle

\begin{large}

%\tableofcontents

\noindent{\bf Abstract:} Let $K$ be an absolutely convex infinite-dimensional compact in a Banach space $\mx$.
The set of all bounded linear operators $T$ on $\mx$ satisfying $TK\supset K$ is denoted by $G(K)$. Our starting
point is the study of the closure $WG(K)$ of $G(K)$ in the weak operator topology. We prove that $WG(K)$
contains the algebra of all operators leaving $\overline{\lin(K)}$ invariant. More precise results are obtained
in terms of the Kolmogorov $n$-widths of the compact $K$. The obtained results are used in the study of operator
ranges and operator equations.
\bigskip

\bigskip


\section{Introduction}

Let $K$ be a subset in a Banach space $\mx$. We say (with some
abuse of the language) that an operator $D\in \ml(\mx)$ {\it
covers} $K$, if $DK\supset K$. The set of all operators covering
$K$ will be denoted by $G(K)$. It is a semigroup with a unit since
the identity operator is in  $G(K)$. It is easy to check that if
$K$ is compact then $G(K)$ is closed in the norm topology and,
moreover, sequentially closed in the weak operator topology
(\WOT). It is somewhat surprising that for each absolutely convex
infinite dimensional compact $K$ the \WOT-closure of $G(K)$ is
much larger than $G(K)$ itself, and in many cases it coincides
with the algebra $\ml(\mx)$ of all operators on $\mx$. Our aim is
to understand: how much freedom has an operator which is obliged
to cover a given compact? In a simplest form the question is:
``How large is $G(K)$?''. We answer this question describing the
\WOT-closure $WG(K)$ of $G(K)$ as well as its closure in the
ultra-weak topology (for the case of Hilbert spaces). These
results are obtained in Sections 2--3 for the Banach spaces, and
in more detailed form in Section 4 for Hilbert spaces; they are
formulated in terms of Kolmogorov's $n$-widths of $K$.

In Section 5 we consider a more general object:  the set $G(K_1,K_2)$ of all operators $T$ which have the
property $TK_1 \supset K_2$ where $K_1, K_2$ are fixed convex compacts in Hilbert spaces.

In further sections we apply the obtained results for study of some related subjects: operator ranges (Section
7), operator equations of the form $XAY = B$ (Section 8) and operators with the property
$$\|AXx\| \ge \|Ax\|
\textrm{ for all }x\in \mh$$ where $A$ is a given operator on a Hilbert space $\mh$. Some applications of the
obtained results to the theory of quadratic operator inequalities and operator fractional linear relations will
be presented in a subsequent work. In fact our interest to the semigroups $G(K)$ was initially motivated by
these applications; the relations to other topics became clear for us in the process of the study.
\medskip

\noindent{\bf Notation.} Our terminology and notation of Banach
space theory follows \cite{JL01}. Our definitions of the standard
topologies on spaces of operators follow \cite[Chapter VI]{DS58}.
Let $\mx ,\my $ be Banach spaces. We denote the closed unit ball
of a Banach space $\my $ by $B_\my $, and the norm closure of a
set $M\subset \my $ by $\overline{M}$. We denote the set of
bounded linear operators from $\my $ to $\mx $ by $\ml(\my ,\mx
)$. We write $\ml(\mx)$ for $\ml(\mx,\mx)$. The identity operator
in $\ml(\mx)$ will be denoted by $I$.
\medskip

Throughout the paper we denote by $\lin(K)$ the linear span of a
set $K$, and by $V_K$ the closed subspace spanned by  $K$, that
is, $V_K = \overline{\lin(K)}$\label{Page:V_K}. We denote by
$\mathcal{A}_K$ the algebra of all operators for which $V_K$ is an
invariant subspace. It is clear that $\mathcal{A}_K$ is closed in
the \WOT.
\medskip

\noindent{\bf Remark on related work.} Coverings of compacts by
sets of the form $R(B_{\mz})$ where $\mz$ is a Banach space and
$R\in\ml(\mz,\mx)$ have been studied by many authors, see
\cite{CJ96}, \cite{COS95}, and \cite{FJPS06}. However, the main
foci of these papers are different. In all of the mentioned papers
additional conditions are imposed on $\mz$, or on $R$, or on both
of them, and the main problem is: whether such $R$ exist? In the
context of the present paper existence is immediate, while for us
(as it was mentioned above) the main question is: ``How large is
the set of such operators?''.
\medskip

\noindent{\bf Acknowledgement.} The authors would like to thank
Heydar Radjavi for a helpful discussion and his interest in our
work.
\medskip

We finish the introduction by showing that for non-convex compacts
$K$ the semigroup $G(K)$ can be trivial:

\begin{example}\label{nonconv} There exists a compact $K$ in an infinite dimensional separable Hilbert space $\mh$
such that the only element of $G(K)$ is the identity operator.
\end{example}

\begin{proof} Let $\{e_n\}_{n=1}^\infty$ be an orthonormal basis
in $\mh$; $\{\alpha_n\}_{n=1}^\infty$ be a sequence of real
numbers satisfying $\alpha_n>0$ and
$\lim_{n\to\infty}\alpha_{n+1}/\alpha_n=0$; and
$\{\beta_n\}_{n=1}^\infty$ be a sequence of distinct numbers in
the open interval $\left(\frac12,1\right)$. The compact $K$ is
defined by
\[K=\{0\}\cup\left\{\alpha_n e_n\right\}_{n=1}^\infty\cup \left\{\alpha_n\beta_ne_n\right\}_{n=1}^\infty.\]

Assume that there exists $D\in\ml(\mh)$ such that $D(K)\supset K$
and $D$ is not the identity operator. Let
$M=\{n\in\mathbb{N}:~De_n\ne e_n\}$. Since $\{e_n\}_{n=1}^\infty$
is a basis in $\mh$ and $D$ is not the identity operator, the set
$M$ is nonempty. We introduce the following oriented graph with
the vertex set $M$. There is an oriented edge
$\overrightarrow{nm}$ starting at $n\in M$ and ending at $m\in M$
($n$ can be equal to $m$) if and only if one of the following
equalities holds:
\begin{equation}\label{E:coverpt} \begin{split} &D(\alpha_me_m)=\alpha_ne_n,\;
D(\beta_m\alpha_me_m)=\alpha_ne_n,\\
&D(\alpha_me_m)=\beta_n\alpha_ne_n,\;
D(\beta_m\alpha_me_m)=\beta_n\alpha_ne_n.\end{split}\end{equation}

\noindent{\bf Important observation.} Since the numbers
$\{\beta_n\}_{n=1}^\infty\subset\left(\frac12,1\right)$ are
distinct, the number of edges starting at $n$ is at least $2$ for
each $n\in M$, while there is at most one edge ending at $m\in M$.
\medskip

An immediate consequence of this observation is that there are
infinitely many oriented edges $\overrightarrow{nm}$ with $n<m$,
that is, infinitely many pairs $(n,m)$, $n<m$, for which one of
equalities from \eqref{E:coverpt} holds. Taking into account the
conditions satisfied by $\{\alpha_n\}$ and $\{\beta_n\}$, we get a
contradiction with the boundedness of $D$.
\end{proof}

This example shows that in the general case there is a very strong dependence of the size of the semigroup
$G(K)$ on the geometry of $K$. To relax this dependence we restrict our attention to absolutely convex compacts
$K$.

\section{$\mathcal{A}_K\subset WG(K)$}

\begin{theorem}\label{vklyuch} Let $K$ be an absolutely convex infinite-dimensional compact.
Then $\mathcal{A}_K\subset WG(K)$.
\end{theorem}

\begin{remark} Theorem {\rm \ref{vklyuch}} is no longer true for
finite dimensional compacts. In fact, if $K$ is absolutely convex
finite-dimensional compact, then $A\in G(K)$ implies that $A$
leaves $V_K$ invariant. Since $V_K$ is finite dimensional, the
condition $AK\supset K$ passes to operators from the \WOT-closure.
Thus $WG(K)$ is a proper subset of $\ma_K$. \end{remark}

Let $F$ be a subset of $\mx^*$. We use the notation $F_{\bot}$ for
the {\it pre-annihilator} of $F$, that is, $F_{\bot}:=\{x\in \mx
:~\forall f\in F~ f(x)=0\}$.

\begin{lemma}\label{konechn}
Let $K$ be an absolutely convex infinite-dimensional compact in a
Banach space $\mx$. For each finite dimensional subspace $F\subset
\mx ^*$, each finite dimensional subspace $\my\subset\mx$, and an
arbitrary linear mapping $N:\my\to\mx$ satisfying $N(\my\cap
V_K)\subset \lin(K)$, there is $D\in G(K)$ satisfying the
condition
\begin{equation}\label{E:annihilator} Dx - Nx \in F_{\bot} ~~~\forall x\in \my .\end{equation}
\end{lemma}

\begin{proof} Note first of all that it suffices to prove the lemma under an additional assumption that $\my
\subset V_K$. Indeed, suppose that it is done, then in the general
case we choose a complement $\my_1$ of $\my\cap V_K$ in $\my$ and
choose a complement $\mx_1$ of $\my_1$ in $\mx$ that contains
$V_K$. By our assumption there is an operator $D\in \ml(\mx_1)$
with $DK\supset K$ and $Dx - Nx \in F_{\bot}\cap \mx_1 ~~\forall
x\in \my\cap V_K$. It remains to extend $D$ to $\mx$ setting $Dx =
Nx$ for $x\in \my_1$.

So we assume that $\my \subset  V_K$. For brevity denote $\lin(K)$
by $\mz$.

 Let $P:\mx \to \mx $ be a projection of
finite rank, such that $P\mx\supset \my + N\my$, $(I-P)\mx \subset
F_{\bot}$ and $\dim(P\mz)\ge \dim F+\dim \my $. The last condition
can be satisfied since $K$ is finite-dimensional.
\medskip

The conditions $N(\my\cap V_K)\subset\lin(K)$ and $\my\subset V_K$
imply that the subspace $N\my $ is contained in $\mz \cap P\mx $.
The space $\ker (P)\cap \mz$ has finite codimension in $\mz $.
Therefore there exists a complement $L$ of $\ker (P) \cap \mz$ in
$\mz $ such that $L\supset N\my$.
\medskip

We have $PL=P\mz$ and $L\cap (I-P)\mx =\{0\}$. Since the subspace
$(I-P)\mx$ has finite codimension in $\mx $, we can find a
subspace $M\supset L$, which is a complement of $(I-P)\mx$ in $\mx
$. Let $Q_M:\mx \to M$ be the projection onto $M$ with the kernel
$\ker(P)=(I-P)\mx$, and let $M_0$ be the complement of $L$ in $M$.
Since $\mz=(\ker P\cap \mz)\oplus L$ and $Q_M(L)=L$, we have
$Q_M(\mz)=L\subset\mz$.
\medskip

To introduce an operator $D\in\ml(\mx )$ it suffices to determine
its action on $\ker P$ and on $M$. We do it in the following way:
\medskip

(a) The restriction of $D$ to $\ker P$ is a multiple $\lambda
I_{\ker P}$ of the identity operator, where $\lambda$ is chosen in
such a way that $(I-Q_M)K \subset \frac{\lambda}{2}K$ (such a
choice is possible because $(I-Q_M)K$ is a compact subset of $\mz
= \cup_{n\in \mathbb{N}}n K$).
\medskip

(b) The restriction of $D$ to $M$ is defined in three `pieces':

\begin{itemize}

\item $D|_{M_0}=0$.

\item Now we define the restriction of $D$ to $Q_M(\my)$.
Observe that $Q_M(\my)\subset L$. This follows from
\[\my \subset V_K\subset L\oplus (I-P)\mx.\]
In addition $Q_M|_{\my }$ is an isomorphism, because $\my\subset
P\mx $, and $P\mx $ and $M$ are complements of the same subspace.
Because of this, the operator $D|_{Q_M(\my)}$ given by
\[D(Q_M(x))=Nx+\alpha S(Q_M(x))\hbox{ for }x\in \my\]
is well-defined, where $\alpha\in\mathbb{R}$ and $S$ is an
isomorphism of $Q_M(\my)$ into $F_\bot\cap L$. Such isomorphisms
exist because $\dim L\ge \dim(PL)=\dim(P\mz)$, and we assumed that
$\dim(P\mz)\ge\dim F+\dim\my$. Now we choose $\alpha$ to be so
large that the image of $K\cap Q_M(\my)$ covers a `large' multiple
of the intersection of $Q_M(K)$ with the space onto which it maps.
This is possible because zero has non-empty interior in $K\cap
Q_M(\my)$ and $Q_M(K)$ is compact.

\item We define $D$ on the complement of $Q_M(\my)$ in $L$ as
a `dilation' operator onto some complement of the $D(Q_M(\my))$ in $L$. The number $\alpha$ and the dilation are
selected in such a way that
\begin{equation}\label{E:iiLemma} D(K\cap L)\supset 2Q_M(K ).
\end{equation}
\end{itemize}
To see that it is possible recall that $Q_M(K)\subset L$ and,
since $L$ is finite-dimensional, the set $K\cap L$ contains a
multiple of the unit ball of $L$.
\medskip

It remains to verify that $D$ satisfies the conditions
\eqref{E:annihilator} and $DK\supset K$.
\medskip

Condition \eqref{E:annihilator}. Let $x\in \my $, then
$x=Q_Mx+(I-Q_M)x$. Therefore
\[Dx=Nx+\alpha S(Q_M(x))+\lambda (I-Q_M)x.\]
Let $f\in F$. We get:
\[f(Dx)=f(Nx)+\alpha f(S(Q_M(x)))+\lambda
f(I-Q_M)x=f(Nx),\] where we use the following facts: (a) The image
of $S$ is in $F_\bot$; (b) $(I-Q_M)\mx=(I-P)\mx\subset F_{\bot}$.
\medskip

Condition $DK\supset K$. Let $x\in K$. Then $x =Q_Mx+(I-Q_M)x$.
The condition (\ref{E:iiLemma}) implies that there exists
$v\in\frac12(L\cap K)$ such that $Dv=Q_Mx$. The choice of
$\lambda$ implies that $w={\frac1\lambda(I-Q_M)x}$ satisfies
$w\in\frac{1}{2} K $. Let $z=v+w$. It is clear that $z\in K$. We
need to show that $Dz=x$. We have
\[\begin{split}Dz&=Dv+Dw=Dv+D \left(\frac1\lambda(I-Q_M)x\right)\\
&=Q_Mx+\lambda \left(\frac1\lambda(I-Q_M)x\right)=x.\end{split}\]
(We use the fact that $(I-Q_M)\mx \subset\ker P$.)
\end{proof}

\begin{proof}[Proof of Theorem \ref{vklyuch}] Let $T\in\ma_K$, and $$\mju=\{E\in\ml(\mx):~
\forall i\in\{1,\dots, n\}~ |f_i(Ex_i)|<\ep\}$$ be a \WOT-neighborhood of $0$ in $\ml(\mx)$, where
$n\in\mathbb{N}$, $\ep>0$, $\{f_i\}_{i=1}^n\in \mx^*$ and $\{x_i\}_{i=1}^n\in\mx$. We need to show that $T+\mju$
contains an operator from $G(K)$ for each choice of $n,\ep, f_i,$ and $x_i$. Let $F=\lin(\{f_i\}_{i=1}^n)$ and
$\my=\lin(\{x_i\}_{i=1}^n)$. Let $\my_1=\my\cap V_K$. Since $T\in\ma_K$, we have $T(\my_1)\subset V_K$. Since
$V_K=\overline{\mz}$, we can find a ``slight perturbation'' $\widetilde T$ of $T$ satisfying $\widetilde
T(\my_1)\subset \mz$. In particular, we can find such $\widetilde T$ in $T+\frac12\mju$. It remains to show that
$\widetilde T+\frac12\mju$ contains an operator $D$ from $G(K)$. \medskip

It is clear that each operator $S$ satisfying \[\forall
x\in\my\quad Sx-\widetilde Tx\in F_\bot\] is in $\widetilde
T+\frac12\mju$. Now the existence of the desired operator $D$ is
an immediate consequence of Lemma \ref{konechn} applied to
$N=\widetilde T$.
\end{proof}

\begin{corollary}\label{sled} If $V_K = \mx$, then $WG(K) = \ml(\mx )$.
\end{corollary}

\section{Application of Kolmogorov $n$-widths to estimates of the `size' of $WG(K)$ from above}

We are going to use the notion of Kolmogorov $n$-width. In this
respect we follow the terminology and notation of the book
\cite[Chapter II]{Pin85}. Let $\mz $ be a subset of a Banach space
$\mx $ and $x\in \mx $. The {\it distance from $x$ to $\mz $} is
defined as
\[E(x,\mz )=\inf\{||x-z||:~z\in \mz \}.\]

\begin{definition} {\rm Let $K$ be a subset of a Banach space $\mx $, $n\in\mathbb{N}\cup\{0\}$.
The {\it Kolmogorov $n$-width} of $K$ is given by
\[d_n(K)=\inf_{\mx _n}\sup_{x\in K}E(x,\mx _n),\] where the infimum is
over all $n$-dimensional subspaces. If
\[d_n(K)=\sup_{x\in K}E(x,\mz )\]
and $\mz \subset \mx $ is an $n$-dimensional subspace, then $\mz $
is said to be an {\it optimal} subspace for $d_n(K)$.}
\end{definition}

\begin{lemma}\label{L:cover} Let $K$ and $K_0$ be two subsets in a Banach space $\mx $ and $D\in \ml(\mx )$ be such that
$D(K_0)\supset K$. Then $d_n(K)\le||D||d_n(K_0)$ for all $n\in\mathbb{N}\cup\{0\}$.
\end{lemma}

\begin{proof}  Let $\mz \subset \mx $ be an $n$-dimensional
subspace. Then $D\mz \subset \mx $ is a subspace of dimension $\le
n$ and $E(Dx,D\mz )\le ||D|| E(x,\mz )$. The conclusion follows.
\end{proof}

\begin{lemma}\label{L:codim1} Let $K$ be a bounded subset in a Banach space $\mx$. If $K_0=K\cap L$, where $L$ is a
closed linear subspace in $\mx$ which does not contain $K$, then
there exists a constant $0<C<\infty$ such that $d_n(K_0)\le
Cd_{n+1}(K)$ for all $n\in\mathbb{N}\cup\{0\}$.
\end{lemma}

\begin{proof} It is well-known (see \cite[p.~10]{Pin85}) that a bounded set $K$
is compact if and only if $\lim_{n\to\infty}d_n(K)=0$. Therefore
it suffices to consider the case when $K$ is compact. It is
clearly enough to consider the case when $L$ is a subspace of
codimension $1$. Let $L=\ker\nu$ where $\nu\in\mx^*$. We may
assume without loss of generality that the norm of the restriction
of $\nu$ to $\lin K$ satisfies $||\nu|_{{\rm lin}(K)}||=1$. For
each $n\in\mathbb{N}\cup\{0\}$ let $L_n\subset\mx$ be a subspace
of dimension $n$ satisfying $\sup_{x\in K}E(x,L_n)\le
2d_n(K)$.\medskip

First we show that there exists $N\in\mathbb{N}$ such that
$||\nu|_{L_M}||>\frac12$ for all $M\ge N$. Let $0<\ep<1$ and let
$x_i\in K$ and scalars $a_i$ $(i=1,\dots,k)$ be such that the
vector ${h=\sum_i a_ix_i}$ satisfies $||h||=1$ and $\nu(h)>1-\ep$.
Let $\delta>0$ be such that $\delta||\nu||\sum|a_i|<\ep$. Let $N$
be such that for $M\ge N$ we have $d_M(K)<\delta/2$. Then for
$M\ge N$ there exist $y_i\in L_M$ such that $||x_i-y_i||<\delta$.
Therefore the vector $g:=\sum_i a_iy_i$ satisfies
$||\nu||\cdot||g-h||<\ep$ and $g\in L_M$. Choosing appropriate
$\ep$ and $\delta$ we get $||\nu|_{L_M}||>\frac12$.
\medskip

Let $M\ge N$ and let $L_{M,0}=L_M\cap\ker\nu$. Let $x\in K_0$. We
are going to show that $E(x,L_{M,0})<(2||\nu||+1)2d_M(K)$. By the
definition of $L_M$ there is $y\in L_M$ such that $||x-y||\le
2d_M(K)$. Since $\nu(x)=0$, we have $|\nu(y)|\le||\nu|| 2d_M(K)$.
Since $||\nu|_{L_M}||>\frac12$, we conclude that $E(y, L_{M,0})<
4||\nu|| d_M(K)$. Therefore $E(x,L_{M,0})<(2||\nu||+1)2d_M(K)$. It
is clear that $\dim L_{M,0}=M-1$. Thus for $M\ge N$ we have
$d_{M-1}(K_0)<(2||\nu||+1)2d_M(K)$. The conclusion follows.
\end{proof}

\begin{definition} {\rm Let $\{a_n\}$ be a non-increasing sequence of
positive numbers satisfying $\lim_{n\to\infty}a_n=0$. We say that
$\{a_n\}$ is {\it lacunary} if}
\begin{equation}\label{usl-bystr}
\liminf_{n\to\infty} \frac {a_{n+1}}{a_n} = 0.
\end{equation}
\end{definition}

\begin{lemma}\label{bystr}
If the sequence  $\{d_n(K)\}_{n=1}^\infty$ is lacunary, then
$G(K)\subset \mathcal{A}_K$.
\end{lemma}

\begin{proof} Let
$R\in\ml(\mx )$ be such that $RV_K$ is not contained in $V_K$. We
have to show  that $RK$ does not contain $K$. Assume the contrary.

It follows from our assumption that $R^{-1}(V_K)$ is a proper
subspace of $V_K$ and $R(K_0)\supset K$ where $K_0=K\cap
R^{-1}(V_K)$ is a proper section of $K$.
\medskip

By Lemma \ref{L:cover} we get $d_n(K)\le ||R||d_n(K_0)$ for all
$n\in\mathbb{N}\cup\{0\}$. By Lemma \ref{L:codim1} we get
$d_n(K_0)\le Cd_{n+1}(K)$ for some $0<C<\infty$ (which depends on
$K$ and $K_0$, but not on $n$) and all $n\in\mathbb{N}\cup\{0\}$.
We get $d_{n+1}(K)\ge (C||R||)^{-1}d_n(K)$, hence the sequence
$\{d_n(K)\}_{n=1}^\infty$ is not lacunary. We get a contradiction.
\end{proof}

\begin{remark} The assumptions of convexity and symmetry of $K$ are not needed in Lemmas
{\rm\ref{L:cover}}, {\rm\ref{L:codim1}}, and {\rm\ref{bystr}}.
\end{remark}

Combining Theorem \ref{vklyuch} and Lemma \ref{bystr} we get

\begin{theorem} If an absolutely convex compact $K$ is such that the sequence $\{d_n(K)\}_{n=1}^\infty$ is lacunary, then
$WG(K)=\mathcal{A}_K$.
\end{theorem}

\section{Covering of ellipsoids}

\subsection{$s$-numbers}

Now we restrict our attention to the Hilbert space case, that is,
we consider  sets $K$ of the form $A(B_{\mh})$ where $A$ is an
infinite-dimensional bounded compact operator from a Hilbert space
$\mh$ to a Hilbert space $\mh_1$. Such sets are called {\it
ellipsoids}.
\medskip

\noindent{\bf Note.} We continue using the Banach space theory
notation and terminology. In particular, unless explicitly stated
otherwise, by $A^*$ we mean the Banach-space-theoretical conjugate
operator. It does not seem that anything will be gained if we
introduce Hilbert-space duality, but it can cause some confusion
when we apply Banach space case results for Hilbert spaces.

\begin{remark} Many of the results below are true for $A(B_\mh)$
with non-compact $A$ and usually the corresponding proofs are much
simpler. We restrict our attention to the compact case.
\end{remark}

\begin{definition} {\rm (See \cite[Chapter II, \S2]{GK69}) The
eigenvalues of the operator $(E^*E)^{1/2}$ (where $E^*$ is the
conjugate in the Hilbert space sense) are called the {\it
$s$-numbers} of the operator $E$. Notation:
$\{s_n(E)\}_{n=1}^\infty$.}
\end{definition}

With this notation we have the following equalities for
$n$-widths: $$d_n(A(B_{\mh}))=s_{n+1}(A)$$ (see \cite[Theorem
2.2, p.~31]{GK69}).\medskip

For ellipsoids we have a converse to the Lemma \ref{L:cover}.

\begin{lemma}\label{conver:cover}
If $K_0,K$ are ellipsoids in Hilbert spaces $\mh_1$, $\mh_2$,
respectively, and $d_n(K)\le C d_n(K_0)$ for some $C>0$ and all
$n\in\mathbb{N}\cup\{0\}$, then there is an operator $D\in
\ml(\mh_1,\mh_2)$ such that $DK_0\supset K$ and $\|D\|\le C$.
\end{lemma}

\begin{proof}
The result follows from the so-called {\it Schmidt expansion} of a
compact operator (see \cite[p.~28]{GK69}), which implies that
\[K= A(B_{\mh}) = \left\{\sum_{n=1}^\infty\alpha_ns_n(A)h_n:~
\{\alpha_n\}_{n=1}^\infty\in\ell_2,~ \{h_n\}_{n=1}^\infty\hbox{ is an orthonormal sequence}\right\}\] and
\[K_0= B(B_{\mh}) = \left\{\sum_{n=1}^\infty\alpha_ns_n(B)g_n:~
\{\alpha_n\}_{n=1}^\infty\in\ell_2,~ \{g_n\}_{n=1}^\infty\hbox{ is
an orthonormal sequence}\right\}.\] It is easy to see that there
is a bounded linear operator $D$ which maps $g_n$ onto $Ch_n$, and
that this operator satisfies the conditions $D(K_0)\supset K$,
$\|D\| \le C$.
\end{proof}

\begin{remark}\label{R:operator} The proof of Lemma {\rm\ref{conver:cover}} shows that the desired operator $D$
can be constructed as an operator whose restriction to $V_{K_0}$
is a multiple of a suitable chosen bijective isometry between
$V_{K_0}$ and $V_K$, extended to $\mh_1$ in an arbitrary way.
\end{remark}

Known results on $s$-numbers imply the following lemma.

\begin{lemma}\label{L:FinCodEllips} Let $K$ be an ellipsoid in a Hilbert space $\mh$ such
that $\{d_n(K)\}_{n=0}^\infty$ is not lacunary. Let $K_0$ be the
intersection of $K$ with a closed linear subspace of finite
codimension. Then there exists $\delta>0$ such that
$d_n(K_0)\ge\delta d_n(K)$ and a bounded linear operator
$Q:\overline{\lin(K_0)}\to\overline{\lin(K)}$ satisfying
$Q(K_0)\supset K$.
\end{lemma}

\begin{proof}  Let $A:\mh\to\mh$ be a compact operator satisfying $K = A(B_{\mh})$. The sequence
$\{d_n(K_0)\}_{n=0}^\infty$ is the sequence of $s$-numbers of a
restriction of $A$ to a subspace of finite codimension. This
sequence is, in turn, the sequence of $s$-numbers of an operator
of the form $A+G$, where $G$ is an operator of finite rank.
\medskip

It is known \cite[Corollary 2.1, p.~29]{GK69} that $s_n(A+G)\ge
s_{n+r}(A)$, where $r$ is the rank of $G$. Combining this
inequality with the assumption that the sequence
$\{s_n(A)\}_{n=1}^\infty$ is not lacunary, we get the desired
inequality.
\medskip

The last statement of the lemma follows from Lemma
\ref{conver:cover}.
\end{proof}

\subsection{\WOT}

\begin{theorem}\label{T:Hilbert+K} If $\mh$ is a Hilbert space, $K\subset \mh$ is an ellipsoid
and the sequence $\{d_n(K)\}$ is not lacunary, then $WG(K)=
\ml(\mh)$.
\end{theorem}

The proof of Theorem \ref{vklyuch} shows that to prove Theorem
\ref{T:Hilbert+K} it suffices to prove the following lemma (this
can also be seen from the definition of \WOT).

\begin{lemma}\label{two-slow}
Let $K$  be an ellipsoid in a Hilbert space $\mh$. Suppose that
the sequence $\{d_n(K)\}$ is non-lacunary. Then for each
finite-dimensional subspace $\my\subset \mh$ and each linear
mapping $N:\my \to \mh$, there is an operator $D$ satisfying
conditions: $Dy = Ny$ for all $y\in \my$, and $DK\supset K$.
\end{lemma}

\begin{proof}
Let $\mz = {\my}^{\bot}$ and $K_0 = K\cap \mz$. By Lemma
\ref{L:FinCodEllips} there is an operator $E$ from $\mz$ to $\mh$
with $EK_0\supset K$. Extend it to an operator $D: \mh\to \mh$
setting $Dy = Ny$ on $\my$.
\end{proof}

\begin{remark}\label{R:SOT} One can see from the proof of Lemma {\rm\ref{two-slow}} that under the stated conditions
the closure of $G(K)$ in the strong operator topology coincides
with $\ml(\mh)$. \end{remark}

\begin{corollary}\label{klas} Let $K$ be an ellipsoid in a Hilbert
space $\mh$. Then:
\begin{itemize}
\item[{\bf (1)}] $WG(K)=\ma_K$ if the sequence
$\{d_n(K)\}_{n=0}^\infty$ is lacunary. \item[{\bf (2)}]
$WG(K)=\ml(\mh)$ if the sequence $\{d_n(K)\}_{n=0}^\infty$ is not
lacunary.
\end{itemize}
\end{corollary}

\subsection{Ultra-weak topology}

It turns out that Theorem \ref{T:Hilbert+K} remains true if we
replace closure in the weak operator topology, by a closure in a
stronger topology, usually called {\it ultra-weak topology}. This
topology on $\ml(\mh)$ is defined as the weak$^*$ topology
corresponding to the duality $\ml(\mh)=(C_1(\mh))^*$, where
$C_1(\mh)$ is the space of nuclear operators. (Necessary
background can be found in \cite[Chapter II]{Tak79}, unfortunately
the terminology and notation there is different, the ultra-weak
topology is called {\it $\sigma$-weak topology}, see
\cite[p.~67]{Tak79}). Ultra-weak and strong operator topologies
are incomparable, for this reason our next result does not follow
from Remark \ref{R:SOT}.

\begin{theorem}\label{T:ultra} If $K$ is an ellipsoid in a Hilbert space $\mh$ and the sequence $\{d_n(K)\}_{n=1}^\infty$ is
not lacunary, then the ultra-weak closure of $G(K)$ coincides with
$\ml(\mh)$.
\end{theorem}

\begin{proof} Let $\{T_i\}_{i=1}^m$ be a finite collection of operators in
$C_1(\mh)$ and $R\in\ml(\mh)$. It suffices to show that there is
$D\in\ml(\mh)$ satisfying
\begin{equation}\label{E:D}
\tr(DT_i)=\tr(RT_i)\hbox{ for } i=1,\dots,m\hbox{ and }DK\supset
K.
\end{equation}
It is clear that we may assume that the operators $T_i$ are
linearly independent.

\begin{lemma}\label{L:surjective} If $\{T_i\}_{i=1}^m$ are linearly independent, then
there exists a finite rank projection $P\in\ml(\mh)$ such that the
mapping
\[\omega: \ml(\mh)\to \mathbb{R}^m\]
given by
\[\omega(U)=\left\{\tr(UPT_i)\right\}_{i=1}^m\] is surjective.
\end{lemma}

\begin{proof} We have to prove that there is a finite rank projection $P$ such
that the operators $PT_i$ are linearly independent (in this case
the mapping $\omega$ will be surjective). Using induction we may
suppose that $P_0T_1,...,P_0T_{m-1}$ are linearly independent for
some $P_0$. Consider the set $M_0$ of those finite rank
projections $P$ which commute with $P_0$ and satisfy $\im
P\supset\im P_0$. We claim that there exists $P\in M_0$ such that
$PT_1,\dots,PT_m$ are linearly independent.
\medskip

Assume contrary, then for each $P\in M_0$, one can find
$\lambda_1(P),...,\lambda_{m-1}(P) \in \mathbb{C}$ satisfying
$PT_m = \sum_{k=1}^{m-1}\lambda_k(P)PT_k$ (using the definition of
$M_0$ it is easy to get a contradiction if $PT_1,\dots,PT_{m-1}$
are linearly dependent). Our next step is to show that the numbers
$\{\lambda_k(P)\}_{k=1}^{m-1}$ do not depend on $P$. In fact, for
any $P_1,P_2\in M_0$ we have $\sum_{k<m}(\lambda_k(P_1)
-\lambda_k(P_2))P_0T_k = 0$. So let $\{\lambda_k\}_{k=1}^{m-1}$ be
such that $\lambda_k(P) = \lambda_k$ for all $P\in M_0$. Then the
operator $T = T_m - \sum_{k<m}\lambda_kT_k$ has the property that
$PT = 0$ for all $P\in M_0$. It is easy to see that this implies
$T = 0$. We get a contradiction with the linear independence of
$\{T_k\}_{k=1}^m$.
\end{proof}

We complete the proof of the theorem by showing the existence of
$D$ satisfying (\ref{E:D}). \medskip

\noindent {\bf 1.} We define $D$ on $\ker P$ as in Theorem
\ref{T:Hilbert+K}. This definition implies that the condition
$DK\supset K$ is satisfied.
\medskip

\noindent {\bf 2.} To show that the condition
$\tr(DT_i)=\tr(RT_i)$, $i=1,\dots,m$, is satisfied it suffices to
show the existence of $U\in\ml(\mh)$ satisfying
\begin{equation}\label{E:U}\tr((UP+D(I-P))T_i)=\tr(RT_i)~\forall i=1,\dots,m.\end{equation}
Since the condition (\ref{E:U}) can be rewritten as
$\{\tr(UPT_i)\}_{i=1}^m=\{\tr((R-D(I-P))T_i)\}_{i=1}^m$, where the
right-hand side does not depend on $U$, and the vectors
$\{\tr(UPT_i)\}_{i=1}^m$, $U\in\ml(\mh)$ cover (by Lemma
\ref{L:surjective}) the whole space $\mathbb{R}^m$, the existence
of $U$ satisfying (\ref{E:U}) follows.
\end{proof}

\begin{remark}  It would be interesting to prove an
analogue of Theorem {\rm\ref{vklyuch}} for the ultra-weak
topology.
\end{remark}

\section{Two ellipsoids}\label{S:2ell}

Let $\mh_1,\mh_2$ be two infinite dimensional separable Hilbert
spaces. We consider two ellipsoids, $K_1\subset\mh_1$,
$K_2\subset\mh_2$ and introduce the set
\begin{equation}\label{twocom}
G(K_1,K_2):=\{T\in\ml(\mh_1,\mh_2):~ TK_1\supset K_2\}.
\end{equation}

We are interested in the description of the \WOT-closure of
$G(K_1,K_2)$ which we denote by $WG(K_1,K_2)$. As in the case of
one ellipsoid, the description depends on the behavior of
sequences of Kolmogorov $n$-widths.
\medskip

We start with some simple but useful observations. It is easy to
see that
$$
G(K_2)G(K_1,K_2)G(K_1) \subset G(K_1,K_2).$$

Using this inclusion and elementary properties of \WOT~we get
\begin{equation}\label{module}WG(K_2)WG(K_1,K_2)WG(K_1) \subset WG(K_1,K_2).\end{equation}

Lemmas \ref{L:cover} and \ref{conver:cover} imply that the set
$G(K_1,K_2)$ is non-empty if and only if
\begin{equation}\label{Obolshoe}
d_n(K_2) = O(d_n(K_1)).
\end{equation}
From now on till the end of this section we assume that
\eqref{Obolshoe} is satisfied.

\begin{observation}\label{O:general} By Remark {\rm \ref{R:operator}}, condition \eqref{Obolshoe} implies that
there is an onto isometry $M: V_{K_1}\to V_{K_2}$ and a number
$\alpha\in\mathbb{R}^+$ such that $\alpha M(K_1)\supset K_2$.
Consider decompositions $\mh_1=V_{K_1}\oplus\mr_1$ and
$\mh_2=V_{K_2}\oplus\mr_2$. Let $A_1\in\ml(V_{K_1})$,
$B_1\in\ml(\mr_1,\mh_1)$, $A_2\in\ml(V_{K_2})$,
$B_2\in\ml(\mr_2,\mh_2)$,  and $C:\mr_1\to\mh_2$. Combining
Theorem {\rm \ref{vklyuch}} with \eqref{module} we get that the
composition $(A_2\oplus B_2)(\alpha M\oplus C)(A_1\oplus B_1)$ is
in $WG(K_1,K_2)$, where $A_i\oplus
B_i:V_{K_i}\oplus\mr_i\to\mh_i$, $i=1,2$.
\end{observation}

To state our results on the description of $WG(K_1,K_2)$ we need
the following definitions.

\begin{definition} {\rm The $k^{th}$ {\it left shift of a sequence
$\{a_n\}^\infty_{n=0}$} $(k\ge 0)$ is the sequence
$\{a_{n+k}\}_{n=0}^\infty$.}
\end{definition}

\begin{definition} {\rm Let $\{a_n\}_{n=0}^\infty$ and $\{b_n\}_{n=0}^\infty$ be sequences of non-negative numbers.
We say that $\{a_n\}_{n=0}^\infty$ {\it majorizes}
$\{b_n\}_{n=0}^\infty$ if there is $0<C<\infty$ such that $b_n\le
Ca_n$ for all $n=0,1,2,\dots$.}
\end{definition}

The following theorem is the main result of this section:

\begin{theorem}\label{two:ell}
Let $K_1$ and $K_2$ be infinite dimensional ellipsoids in Hilbert
spaces $\mh_1$ and $\mh_2$. Assume that \eqref{Obolshoe} holds.
Then
\begin{itemize}
\item[{\bf (A)}] If all left shifts of the sequence
$\{d_n(K_1)\}_{n=0}^\infty$ majorize the sequence
$\{d_n(K_2)\}_{n=0}^\infty$, then $WG(K_1,K_2)=\ml(\mh_1,\mh_2)$.

\item[{\bf (B)}] If the $k^{th}$ left shift of $\{d_n(K_1)\}$
majorizes the sequence $\{d_n(K_2)\}$, but the $(k+1)^{th}$ left
shift does not (such cases are clearly possible), then
$WG(K_1,K_2)$ is the set of those operators $T\in\ml(\mh_1,\mh_2)$
for which the image of the space $T(V_{K_1})$ in the quotient
space $\mh_2/V_{K_2}$ is at most $k$-dimensional.
\end{itemize}
\end{theorem}

\begin{proof} {\bf (A)} Observe that to show $WG(K_1,K_2) =
\ml(\mh_1,\mh_2)$ it suffices to find, for an arbitrary
finite-dimensional subspace $\my\in\mh_1$ and an arbitrary
operator $N:\my\to\mh_2$, an operator $D\in\ml(\mh_1,\mh_2)$
satisfying the conditions: $D|_\my=N|_\my$ and $D(K_1)\supset
K_2$. (This condition implies that $G(K_1,K_2)$ is dense in
$\ml(\mh_1,\mh_2)$ even in the strong operator topology.)
\medskip

We find such an operator $D$ in the following way: let $\my^\perp$
be an orthogonal complement of $\my$. The argument of Lemma
\ref{L:FinCodEllips} shows that the sequence $\{d_n(K_1\cap
\my^\perp)\}$ majorizes some left shift of the sequence
$\{d_n(K_1)\}$ and thus, by our assumption, majorizes the sequence
$\{d_n(K_2)\}$. By Lemma \ref{conver:cover} there is a continuous
linear operator $Y:\my^\perp\to\mh_2$ such that $Y(K_1\cap
\my^\perp)\supset K_2$. We let $D|_{\my^\perp}=Y$ and $D|_\my=N$.
It is clear that $D$ has the desired properties.
\medskip

\noindent{\bf (B)} Suppose that the $k^{th}$ left shift of
$\{d_n(K_1)\}$ majorizes $\{d_n(K_2)\}$. Let
$T\in\ml(\mh_1,\mh_2)$ be such that the dimension of the image of
the space $T(V_{K_1})$ in the quotient space $\mh_2/V_{K_2}$ is
$\le k$. We show that $T\in WG(K_1,K_2)$.
\medskip

Let $F$ be a finite subset of $\mh_2^*$ and $\my$ be a finite
subset of $\mh_1$. It suffices to show that there exists $D\in
G(K_1,K_2)$ satisfying $f(Dy)=f(Ty)$ for each $y\in\my$ and each
$f\in F$. With this in mind, we may assume that $F$ and $\my$ are
finite dimensional subspaces. Also, we may assume that $\my$ is a
subspace of $V_{K_1}$, because we may let the restriction of $D$
to the orthogonal complement of $V_{K_1}$ be the same as the
restriction of $T$.
\medskip

We decompose $F$ as $F_O\oplus F_V$, where $F_O=F\cap
V_{K_2}^\perp$. It is easy to check that the assumption on $T$
implies that $(T^*F_O)_\bot\cap V_{K_1}$ is of codimension at most
$k$ (if it is of codimension $\ge k+1$, then we can find $k+1$
vectors $x_i\in V_{K_1}$ and $k+1$ functionals $x_j^*$ in $F_O$
such that $T^*x_j^*(x_i)=\delta_{i,j}$, but then
$x^*_j(Tx_i)=\delta_{ij}$ shows that $\{Tx_i\}$ is a family of
$k+1$ vectors whose images in $\mh_2/V_{K_2}$ are linearly
independent, contrary to our assumption).
\medskip

Now we decompose $\my=\my_1\oplus\my_2$, where $\my_1= \my\cap
(T^*F_O)_\bot$. We let $D|_{\my_2}=T|_{\my_2}$.  Our next step is
to find a suitable definition of the restriction of $D$ to
$(T^*F_O)_\bot\cap V_{K_1}$. To this end we need the following
modification of Lemma \ref{L:FinCodEllips}, which can be proved
using the same argument and Remark \ref{R:operator}.

\begin{lemma}\label{L:4.5modif} Let $K_1$ and $K_2$ be
ellipsoids in Hilbert spaces $\mh_1$ and $\mh_2$, respectively.
Suppose that the $k^{th}$ left shift of
$\{d_n(K_1)\}_{n=0}^\infty$ majorizes $\{d_n(K_2)\}_{n=0}^\infty$
and that $K_0$ is the intersection of $K_1$ with a subspace of
$\mh_1$ of codimension $k$. Then there exists an operator
$B:V_{K_0}\to V_{K_2}$ such that $B(K_0)\supset K_2$ and $B$ is a
multiple of a bijective linear isometry of $V_{K_0}$ and
$V_{K_2}$.
\end{lemma}

Applying Lemma \ref{L:4.5modif} we find an operator
$B:((T^*F_O)_\bot\cap V_{K_1})\to V_{K_2}$ which satisfies
$B((T^*F_O)_\bot\cap K_1)\supset K_2$ and is a multiple of a
bijective isometry. Now we modify $B$ using Lemma \ref{konechn},
which we apply for $\mx=V_{K_2}$, $K=K_2$, $\my=B\my_1$,
$N=TB^{-1}|_{B\my_1}$, and $F$ (which is denoted in the same way
in this proof). We denote the operator obtained as a result of the
application of Lemma \ref{konechn} by $H$.
\medskip

We let $D|_{(T^*F_O)_\bot\cap V_{K_1}}=HB$. This formula defines
$D$ on $\my_1$, and this definition is such that
$D|_{\my_1}=T|_{\my_1}$. We extend $D$ to the rest of the space
$\mh_1$ arbitrarily.
\medskip

It is clear that $D$ satisfies all the assumptions. Thus $T\in
WG(K_1,K_2)$.
\medskip

Now we suppose that the $(k+1)^{th}$ left shift of $\{d_n(K_1)\}$
does not majorize $\{d_n(K_2)\}$ and show that if $T$ is an
operator for which $T(V_{K_1})$ contains $k+1$ vectors whose
images in the quotient space $\mh_2/V_{K_2}$ are linearly
independent, then $T\notin WG(K_1,K_2)$.
\medskip

Using the standard argument we find $v_1, \dots, v_{k+1}\in
V_{K_1}$, functionals $f_1,\dots,f_{k+1}\in\mh_2^*$,
  and $\ep>0$ such that
any $D\in\ml(\mh_1,\mh_2)$ satisfying $|f_j(Dv_i-Tv_i)|<\ep$,
$i,j=1,\dots,k+1$, satisfies the condition: $D(V_{K_1})$ contains
$k+1$ vectors whose images in $\mh_2/V_{K_2}$ are linearly
independent. It remains to show that such operators $D$ cannot
satisfy $DK_1\supset K_2$.
\medskip

In fact the condition about $k+1$ linearly independent vectors
implies that $D^{-1}(V_{K_2})\cap V_{K_1}$ is a subspace of
$V_{K_1}$ of codimension at least $k+1$.
\medskip

Therefore $K_2$ is covered by a section $K_0$ of $K_1$ of
codimension $k+1$. On the other hand, by Lemma \ref{L:codim1}, the
sequence of $n$-widths of $K_0$ is majorized by the $(k+1)^{th}$
left shift of $\{d_n(K_1)\}_{n=1}^\infty$. By Lemma \ref{L:cover},
we get a contradiction with our assumption.
\end{proof}

\begin{corollary}\label{C:non-lac} If $\{d_n(K_1)\}_{n=0}^\infty$ is non-lacunary and the condition
\eqref{Obolshoe} is satisfied, then $WG(K_1,K_2) =
\ml(\mh_1,\mh_2)$.
\end{corollary}

In fact, if $\{d_n(K_1)\}_{n=0}^\infty$ is non-lacunary, it is
majorized by each of its left shifts, and hence the assumption of
Theorem \ref{two:ell}{\bf (A)} is satisfied.

\begin{remark} In the case where $\{d_n(K_1)\}_{n=0}^\infty$ is lacunary
both the situation described in Theorem {\rm\ref{two:ell}{\bf
(A)}} and the situation described in Theorem {\rm\ref{two:ell}{\bf
(B)}} can occur.
\end{remark}

Similarly to the case of one compact we introduce
$$\ma_{K_1,K_2}:=\{T\in\ml(\mh_1,\mh_2):~ TV_{K_1}\subset
V_{K_2}\}.$$ The following is a special case of Theorem
{\rm\ref{two:ell}{\bf (B)}} corresponding to the case $k=0$:

\begin{corollary}\label{twoquick}
Let $K_1,K_2$ be ellipsoids with
\begin{equation}\label{domin}
\liminf \frac{d_{n+1}(K_1)}{d_n(K_2)} = 0.
\end{equation}
Then $WG(K_1,K_2)=\ma_{K_1,K_2}$.
\end{corollary}

\begin{remark}  Note that the combination of the assumption \eqref{Obolshoe} and the condition \eqref{domin}
imply that the sequences $\{d_n(K_1)\}_{n=1}^\infty$ and
$\{d_n(K_2)\}_{n=1}^\infty$ are both lacunary. Indeed, $d_k(K_2)
\le C d_k(K_1)$ implies
\[\frac{d_{n+1}(K_1)}{d_n(K_2)}\ge \frac{d_{n+1}(K_1)}{Cd_n(K_1)}\hbox{ and }
\frac{d_{n+1}(K_1)}{d_n(K_2)}\ge \frac{d_{n+1}(K_2)}{Cd_n(K_2)}.\]
Therefore \eqref{domin} implies that $\{d_n(K_1)\}_{n=1}^\infty$
and $\{d_n(K_2)\}_{n=1}^\infty$ are lacunary.
\end{remark}

Analysis of all possible cases in Theorem \ref{two:ell} implies
also the following:

\begin{corollary}\label{raznye} If $K_1$ and  $K_2$ are ellipsoids for which $V_{K_i}
= \mh_i$ for $i=1,2$, and \eqref{Obolshoe} is satisfied, then
$WG(K_1,K_2) = \ml(\mh_1,\mh_2)$.
\end{corollary}

\section{Covering with compact operators}

Here we discuss the problem of covering an ellipsoid $K_2$ by the image of an ellipsoid $K_1$ via a compact
operator. Let $CG(K_1,K_2)$ be the set of all compact operators $T$ satisfying the condition $TK_1\supset K_2$.

Let us begin with an analogue of Lemma \ref{L:cover}.

Note that the widths $d_n(K)$ of a compact subset $K$ in a Banach
space $\mx$ can change if we consider $K$ as a subset of a
subspace $\my\subset\mx$ that contains $K$. Let us denote by
$\tilde{d}_n(K)$ the $n$-width of $K$ considered as a subset of
$V_K$ (recall that $V_K=\overline{\lin K}$, so we choose the
minimal subspace and obtain maximal widths).

\begin{lemma}\label{comp-comp}
Let $\mx$ and $\my$ be Banach spaces, $K$ be a compact set in
$\mx$ and $T:\mx\to\my$ be a compact operator. Then
$\tilde{d}_n(TK)/\tilde{d}_n(K)\to 0$ as $n\to\infty$.
\end{lemma}

\begin{proof} We may assume that $\mx = V_K$. By the definition of $\tilde{d}_n$, for each $n\in\mathbb{N}\cup\{0\}$
and $0<\ep<1$, there exists an $n$-dimensional subspace
$\mx_n\subset\mx$ such that
\begin{equation}\label{E:almostApprox} K\subset \mx_n+d_n(K)(1+\ep)B_\mx.\end{equation}

Therefore \begin{equation}\label{E:T_approx}TK\subset
T\mx_n+d_n(K)(1+\ep)TB_\mx.\end{equation}

Now we show that for each $\delta>0$ there is $N\in\mathbb{N}$ such that
\begin{equation}\label{E:delAppr}TB_\mx\subset T\mx_n+\delta B_\my\hbox{ for }n\ge N.\end{equation}

In fact, since $\overline{TB_\mx}$ is compact, it has a finite
$\delta/3$-net $\{y_i\}_{i=1}^t\subset TB_\mx$. Since
$TB_\mx\subset TV_K$, the vectors $y_i$ can be arbitrarily well
approximated by linear combinations of vectors from $TK$. Let $M$
be the maximum absolute sum of coefficients of a selection of such
$\delta/3$-approximating linear combinations. Let $N$ be such that
for $n\ge N$ we have $d_n(K)\le\frac{\delta}{6M||T||}$, and let us
show that \eqref{E:delAppr} holds.

We need to show that for all $y\in TB_\mx$ we have
$\dist(y,T\mx_n)\le\delta$.

Let $j\in\{1,\dots,t\}$ be such that $||y-y_j||<\delta/3$, and let
$\sum_{i=1}^s \alpha_iTx_i$ be such that $x_i\in K$,
$\sum_{i=1}^s|\alpha_i|\le M$, and $||y_j-\sum_{i=1}^s
\alpha_iTx_i||<\delta/3$ . By \eqref{E:almostApprox}, we have
$\dist(x_i,\mx_n)\le d_n(K)(1+\ep)$. Therefore $\dist(\sum_{i=1}^s
\alpha_iTx_i,
T\mx_n)\le\sum_{i=1}^s|\alpha_i|||T||d_n(K)(1+\ep)\le
M||T||\cdot\frac{\delta}{6M||T||}\cdot(1+\ep)<\frac{\delta}3$.
Thus $\dist(y,T\mx_n)<\delta$.

If we combine \eqref{E:T_approx} and \eqref{E:delAppr} we get $d_n(TK)\le(1+\ep)\delta d_n(K)$ for $n\ge N$.
Since $0<\ep<1$ and $\delta>0$ can be chosen arbitrarily, the statement follows.
\end{proof}

\begin{remark}\label{comp-hil}
Note that if $\mx$ is a Hilbert space, then $\tilde{d}_n(K) =
d_n(K)$. Indeed, in this case we may assume that $\mx_n\subset
V_K$. Such subspace can be found as the orthogonal projection to
$V_K$ of any subspace $\mx_n$ satisfying \eqref{E:almostApprox}.
It should be mentioned that in the Hilbert space case a simpler
proof of Lemma {\rm\ref{comp-comp}} is known, see {\rm\cite[Lemma
1]{FR83}}.
\end{remark}

Now we find criteria of non-emptiness of $CG(K_1,K_2)$ for
ellipsoids  $K_1$ and $K_2$. The result can be considered as an
analogue of Lemmas \ref{L:cover} and \ref{conver:cover}.

\begin{lemma} Let $K_1$ and $K_2$ be ellipsoids in Hilbert spaces $\mh_1$ and $\mh_2$, respectively.
There is a compact operator $T$ satisfying $TK_1\supset K_2$ if
and only if
\begin{equation}\label{o-maloe}
d_n(K_2) = o(d_n(K_1)).
\end{equation}

\end{lemma}
\begin{proof}
If there is a compact operator $T$ with $TK_1\supset K_2$ then
(\ref{o-maloe}) follows from Lemma \ref{comp-comp} and Remark
\ref{comp-hil}. Conversely, if (\ref{o-maloe}) holds, then the
existence of a {\it compact} operator $T$ follows from the
argument of Lemma \ref{conver:cover}.
\end{proof}

If the condition (\ref{o-maloe}) is satisfied we say: the sequence
$\{d_n(K_1)\}_{n=1}^\infty$ {\it strictly majorizes}
$\{d_n(K_2)\}_{n=1}^\infty$.

Let us define by $WCG(K_1,K_2)$ the \WOT-closure of $CG(K_1,K_2)$.
We have the following analogue of Theorem \ref{two:ell}:

\begin{theorem}\label{two:ell:comp} {\bf (A)} If all left shifts of the sequence $\{d_n(K_1)\}$
strictly majorize the sequence $\{d_n(K_2)\}$, then $WCG(K_1,K_2)=\ml(\mh_1,\mh_2)$.
\medskip

\noindent{\bf (B)} If the $k^{th}$ left shift of $\{d_n(K_1)\}$ strictly majorizes the sequence $\{d_n(K_2)\}$,
but the $(k+1)^{th}$ left shift does not (such cases are clearly possible), then $WCG(K_1,K_2)$ is the set of
those operators $T\in\ml(\mh_1,\mh_2)$ for which the image of the space $T(V_{K_1})$ in the quotient space
$\mh_2/V_{K_2}$ is at most $k$-dimensional.
\end{theorem}
The proof is a straightforward modification of the proof of
Theorem \ref{two:ell} and we omit it.

\section{Operator ranges}

In this section by a Hilbert space we mean a separable infinite
dimensional Hilbert space. An operator range is the image of a
Hilbert space $\mh_1$ under a bounded operator $A: \mh_1\to
\mh_2$. Operator ranges are actively studied, see \cite{COS95},
\cite{FW71}, \cite{HV06}, \cite{Lon03}, and references therein.
The purpose of this section is to use the results of the previous
section to classify operator ranges. Our results complement the
classification of operator ranges presented in \cite[Section
2]{FW71}.
\medskip

We restrict our attention to images of {\it compact operators of
infinite rank}. The set $A(B_{\mh_1})$ will be called a {\it
generating ellipsoid} of the operator range $A\mh_1$. The same
operator range is the image of infinitely many different
operators, therefore a generating ellipsoid of an operator range
is not uniquely determined. However, the Baire category theorem
implies that if $K_1$ and $K_2$ are generating ellipsoids of the
same operator range, then $cK_1\subset K_2\subset CK_1$ for some
$0<c\le C<\infty$.
\medskip

We say that two sequences of positive numbers are {\it equivalent}
if each of them majorizes the other. The observation above implies
that the equivalence class of the sequence of $n$-widths
$\{d_n(K)\}_{n=0}^\infty$ of a generating ellipsoid of $\my$ is
uniquely determined by an operator range $\my$. We denote this
equivalence class of sequences by $d(\my)$.
\medskip

It is clear that a sequence is lacunary if and only if all of
sequences equivalent to it are lacunary. It is also clear that
left shifts of equivalence classes of sequences are well-defined
as well as the conditions like {\it $d(\my_1)$ majorizes
$d(\my_2)$}. Therefore the following notions are well-defined for
operator ranges: (i) $\my$ is {\it lacunary}; (ii) $\my_1$ {\it
majorizes} $\my_2$. We say that an operator range $\my\subset\mh$
is {\it dense} if $\overline{\my}=\mh$.
\medskip

Results of Section \ref{S:2ell} on covering of one ellipsoid by
another have immediate corollaries for operator ranges. Let
$A_1:\mh\to\mh_1$ and $A_2:\mh\to\mh_2$ be compact operators of
infinite rank and $\my_i = A_i\mh$. Let $\mathcal{R}(\my _1,\my
_2)$ denote the set of all operators $T$ satisfying
\begin{equation}\label{coverrange}
T\my _1 \supset \my _2.
\end{equation}
We write $\mathcal{R}(\my)$ instead of $\mathcal{R}(\my,\my)$. The
\WOT-closure of $\mr(\my_1,\my_2)$ will be denoted by
$\mw\mr(\my_1,\my_2)$.

\begin{corollary}\label{C:OR} Suppose that $\my_1$ majorizes $\my_2$. Then

\begin{itemize}

\item[{\bf (i)}] If all left shifts of $d(\my_1)$ majorize
$d(\my_2)$, then $\mw\mr(\my_1,\my_2)=\ml(\mh_1,\mh_2)$.

\item[{\bf (ii)}] Let $k$ be a non-negative integer. If the
$k^{th}$ left shift of $d(\my_1)$ majorizes $d(\my_2)$, but the
$(k+1)^{th}$ left shift does not, then $\mw\mr(\my_1,\my_2)$ is
the set of those operators $T$ for which the image of $T\my_1$ in
the quotient space $\mh_2/\overline{\my_2}$ has dimension $\le k$.
In particular, if $k=0$, we get: if the first left shift of
$d(\my_1)$ does not majorize $d(\my_2)$, then
$\mw\mathcal{R}(\my_1,\my_2){\overline{\my_1}}\subset
{\overline{\my_2}}$.

\item[{\bf (iii)}] If $\my_1$ is non-lacunary, then
$\mw\mr(\my_1,\my_2)=\ml(\mh_1,\mh_2)$.

\item[{\bf (iv)}] If $\my_1$ and $\my_2$ are dense, then
$\mw\mr(\my_1,\my_2)=\ml(\mh_1,\mh_2)$.

\end{itemize}
\end{corollary}

\begin{proof} To derive {\bf (i)-(iv)} from Theorem \ref{two:ell} and its corollaries
we need two observations:

\begin{itemize}
\item $\mathcal{R}(\my _1,\my _2)$ contains $G(K_1,K_2)$ for any
pair of generating ellipsoids.

\item If $T\in \mathcal{R}(\my_1,\my_2)$ then $TK_1\supset K_2$
for some pair of generating ellipsoids.
\end{itemize}

The first observation immediately implies {\bf (i), (iii), (iv)},
and ``estimates from below'' in {\bf (iv)}. The second observation
shows that for ``estimates from above'' in {\bf (ii)} we can use
the same argument as in Section \ref{S:2ell}.
\end{proof}

One of the systematically studied objects in the theory of
invariant subspaces, see \cite{Foi72,LR86,NRRR77,NRRR79}, is the
algebra $\mathcal{A}(\mathcal{Y})$ of all operators that preserve
invariant a given operator range $\my$. It is known, see
\cite[Theorem 1]{NRRR79}, that if $\my$ is dense, then the
\WOT-closure $\mw\mathcal{A}(\mathcal{Y})$ of
$\mathcal{A}(\mathcal{Y})$ coincides with $\ml(\mh)$. It follows
easily that in general $\mw\mathcal{A}(\mathcal{Y})$ consists of
all operators that preserve the closure $\overline{\my}$ of $\my$.
\medskip

An operator algebra $\mathcal{A}$ is called {\it full} \,if it
contains the inverses of all invertible operators in
$\mathcal{A}$. We call $\mathcal{A}$ {\it weakly full} if for each
invertible operator $T\in \mathcal{A}$, the operator $T^{-1}$
belongs to the \WOT-closure of $\mathcal{A}$. Our next result
shows that for algebras of the form $\mathcal{A}(\mathcal{Y})$
this property depends on $d(\my)$.

\begin{corollary}\label{inverse}

\begin{itemize}
\item[{\bf (i)}]If the closure $\overline{\my}$ of an operator range $\my\subset\mh$ has finite codimension in $\mh$,
then the algebra $\mathcal{A}(\mathcal{Y})$ is weakly full.

\item[{\bf (ii)}] If $\my$ is not lacunary and
$\codim(\overline{\my}) = \infty$, then $\mathcal{A}(\mathcal{Y})$
is not weakly full.

\item[{\bf (iii)}] If $\my$ is lacunary, then $\mathcal{A}(\mathcal{Y})$ is weakly full.
\end{itemize}
\end{corollary}

\begin{proof} {\bf (i)} If $T$ preserves $\my$ then $T\overline{\my}\subset
\overline{\my}$. If $T$ is invertible, then it maps a complement
of $\overline{\my}$ onto a complement of $T(\overline{\my})$. If
$\overline{\my}$ has finite codimension, this implies
$T\overline{\my}=\overline{\my}$. Hence
$T^{-1}\overline{\my}=\overline{\my}$, and $T^{-1}$ is in the
\WOT-closure of $\ma(\my)$.
\medskip

\noindent{\bf (ii)} Let $K$ be a generating ellipsoid of $\my$.
Choose a nonzero vector $y\in\my$ and let $K_0 = K\cap y^{\bot}$.
By Lemma \ref{L:FinCodEllips}, the sequences $\{d_n(K)\}$ and
$\{d_n(K_0)\}$ are equivalent. Using Observation \ref{O:general}
we find an operator $D:V_{K_0}\to V_K$ which satisfies
$D(K_0)\supset K$ and is a (nonzero) multiple of an isometry.
Since $\overline{\my}$ has infinite codimension, we can extend $D$
to an invertible operator $D:\mh\to\mh$. Observe that
$D(\overline{\my}\cap y^\bot)=\overline{\my}$, therefore
$D(y)\notin\overline{\my}$, and thus $D\notin\mw\ma(\my)$. On the
other hand, the inclusion $D(K_0)\supset K$ implies
$D^{-1}\in\ma(\my)$.
\medskip

\noindent{\bf(iii)} If $T\in \mathcal{A}(\mathcal{Y})$ is
invertible, then $T^{-1}\in \mr(\my)$. Since $\my$ is lacunary,
applying Corollary \ref{C:OR} we conclude that $T^{-1}$ preserves
$\overline{\my}$. Therefore $T^{-1}\in
\mw\mathcal{A}(\mathcal{Y})$.
\end{proof}

\section{Bilinear operator equations}

One of the popular topics in operator theory is the study of
linear operator equations $XA = B$ and $AX = B$. We consider here
a ``bilinear operator equation"
\begin{equation}\label{bilin}
XAY = B,
\end{equation}
where operators $A,B$ are given. Its solution is a pair $(X,Y)$ of
operators. We denote the set of all such solutions by $\ms(A,B)$.
For simplicity we restrict our attention to the case when all
operators act on a fixed separable Hilbert space $\mh$. Such a
pair $(X,Y)$ can be found if we fix one of the operators (say $X$)
and solve the obtained linear equation (which has more than one
solution in the degenerate cases only). So the study of the
question ``how many solutions does equation (\ref{bilin}) have?''
reduces to the study of the set of all first components, that is,
the set of those $X$ for which $(X,Y)$ is a solution for some $Y$.
Let us denote this set by $U(A,B)$.

\begin{corollary}\label{eq}
{\rm (i)} The equation is solvable if and only if
\begin{equation}\label{neobh}
s_n(B) = O(s_n(A)).
\end{equation}

{\rm (ii)} Suppose that condition \eqref{neobh} holds. If
operators $A,B$ have dense ranges, or if the range of $A$ is
non-lacunary, then $U(A,B)$ is \WOT-dense in $\ml(\mh )$.

{\rm (iii)} If the range of operator $B$ is not dense and the condition
\begin{equation}\label{neobh-1}
s_n(B) = O(s_{n+1}(A))
\end{equation}
does not hold, then $U(A,B)$ is not \WOT-dense in $\ml(\mh)$.
\end{corollary}

\begin{proof}
Clearly $X\in U(A,B)$ if and only if the equation (\ref{bilin}) is solvable with respect to $Y$. This is
equivalent to the inclusion $XA\mh\supset B\mh$. It remains to apply Corollary \ref{C:OR}.
\end{proof}

If an operator $A$ is not compact then the set is \WOT-dense in
$\ml(\mh)$. Formally this is not a special case of Corollary
\ref{eq}(ii) because $s$-numbers are usually defined for compact
operators only, but the proof in this case along the same lines is
even simpler. In the rest of the section we prove that this result
can be considerably strengthened: if $A$ is not compact then
$\mathcal{S}(A,B)$ itself is dense in $\ml(\mh)\times \ml(\mh)$
with respect to the weak (and even strong) operator topology.

\begin{lemma}\label{inv}
If we are given two linearly independent families $(x_1,...,x_n)$,
$(y_1,...,y_m)$ of vectors in $\mh$, two arbitrary families
$(x_1^{\prime},...,x_n^{\prime})$,
$(y_1^{\prime},...,y_m^{\prime})$ of vectors in $\mh$, and a
number $\epsilon>0$, then there is an invertible operator $V$ with
the properties $\|Vx_i - x_i^{\prime}\| <\epsilon $, $ \|V^{-1}y_j
- y_j^{\prime}\| < \epsilon$ .
\end{lemma}
\begin{proof}
One can choose systems $z_1,...,z_n$ and $w_1,...,w_m$ close to
$(x_1^{\prime},...,x_n^{\prime})$ and, respectively,
$(y_1^{\prime},...,y_m^{\prime})$ in such a way that both systems
\[(x_1,...,x_n, w_1,...,w_m)\hbox{ and  }(y_1,...,y_m,z_1,...,z_n)\]
are linearly independent. Let us define an operator $T$ between
their linear spans by $Tx_i = z_i, Tw_j = y_j$. It is injective
and therefore can be extended to an invertible operator on a
finite dimensional space containing these systems. Clearly an
invertible operator on a finite-dimensional subspace can be
extended to an invertible operator on the whole space (take the
direct sum with the identity operator).
\end{proof}

We denote the group of all invertible operators on $\mh$ by
$\mathcal{G}(\mh)$. {\bf Note.} In this section $A^*$ denotes the
Hilbert space conjugate of an operator $A$.

\begin{lemma}\label{inv2}
If an operator $X$ has dense image and an operator $Y$ has trivial kernel, then the set
$$ {\Gamma}_{X,Y} = \{(XV^{-1},VY): V\in \mathcal{G}(H)\}$$
is dense in $\ml(\mh)\times \ml(\mh)$ with respect to the strong
operator topology {\rm(\SOT)}.
\end{lemma}

\begin{proof}
Let a system $(x_1,...,x_n)$, $(y_1,...,y_m)$,
$(x_1^{\prime},...,x_n^{\prime})$,
$(y_1^{\prime},...,y_m^{\prime})$ and  $\epsilon>0$ be given as
above. The system $\tilde{x}_i = Yx_i$ is linearly independent
since $\ker {Y} = 0$. Since $X\mh$ is dense, there are $z_j$ with
$\|Xz_j - y_j^{\prime}\| < \epsilon/2$. Take $0<\delta <
\frac{\epsilon}{2||X||}$ and choose an invertible operator $V$ as
in Lemma \ref{inv} for the system $(\tilde{x}_1,...,\tilde{x}_n)$,
$(y_1,...,y_m)$, $(x_1^{\prime},...,x_n^{\prime})$,
$(z_1,...,z_m)$ and  $\delta$. The obtained inequalities imply
that $\Gamma_{X,Y}$ is \SOT-dense in $\ml(\mh)\times \ml(\mh)$.
\end{proof}

Any solution $(X,Y)$ of the equation $XY = B$ will be called a {\it factorization} of an operator $B$.

\begin{proposition}\label{factor}
For each operator $B$ in an infinite-dimensional Hilbert space
$\mh$, the set $\mf(B)$ of all its factorizations is \SOT-dense in
$\ml(\mh)\times \ml(\mh)$.
\end{proposition}
\begin{proof}
Let $\mh = \mh_1\oplus\mh_2$ where $\mh_1$ and $\mh_2$ are of the
same dimension as $\mh$. Let $U_1$ and $U_2$ be isometries with
the ranges $\mh_1$ and $\mh_2$, respectively. Then $U_1^*$ and
$U_2^*$ isometrically map $\mh_1$ and $\mh_2$, respectively, onto
$\mh$, also $U_1^*\mh_2 = \{0\}$ and $U_2^*\mh_1 = \{0\}$. We set
$Y = U_1$ and $X = BU_1^* + U_2^*$.

Since $XY = BU_1^*U_1 + U_2^*U_1 = B$, we have $(X,Y)\in \mf(B)$,
and therefore $(XV^{-1},VY) \in \mf(B)$ for each $V\in
\mathcal{G}(H)$. It follows easily from the definition of
operators $X,Y$ that $X\mh = \mh$ and $\ker(Y) = 0$. Applying
Lemma \ref{inv2} we conclude that $\mf(B)$ is SOT-dense in
$\ml(\mh)\times \ml(\mh)$.
\end{proof}

Let us write $A\succ B$  if the set $\ms(A,B)$ of all  solutions of (\ref{bilin}) is \SOT-dense in
$\ml(\mh)\times \ml(\mh)$. For brevity, we will denote by $\overline{\me}^s$ the closure of a subset $\me$ of
$\ml(\mh)\times \ml(\mh)$ with respect to the product of SOT-topologies.

\begin{lemma}\label{trans}
If $A\succ B$ and $B\succ  C$, then $A\succ C$.
\end{lemma}

\begin{proof} For each $(X,Y)\in  \ms(A,B)$ and each $(X_1,Y_1)\in \ms(B,C)$,
one has $(XX_1,Y_1Y)\in \ms(A,C)$. Taking $(X_1,Y_1)\to (I,I)$ we get that $(X,Y)\in \overline{\ms(A,C)}^s$.
Hence $\ml(\mh)\times \ml(\mh)\subset\overline{\ms(A,C)}^s$ and $A\succ C$.
\end{proof}

We proved in Proposition \ref{factor} that $I\succ C$ for all $C$.
So our aim is to show that $A\succ I$ for each non-compact $A$.

\begin{lemma}\label{proj}
If $P$ is a projection of infinite rank, then $P\succ I$.
\end{lemma}
\begin{proof}
Let $U$ be an isometry with $UU^* = P$. Then $(U^*,U)\in \ms(P,I)$. Hence \[(VU^*,UV^{-1})\in\ms(P,I)\hbox{ for
each }V\in \mathcal{G}(\mh).\] It follows that $\overline{\ms(P,I)}^s$ contains all pairs $(M,N)$ with
$N\mh\subset P\mh$, $M(I-P) =0$.

Hence for each $(X,Y)\in \ml(\mh)\times \ml(\mh)$, the pair
$(XP,PY)$ belongs to  $\overline{\ms(P,I)}^s$. Choose a net
$(X_{\lambda},Y_{\lambda})$ in $\ms(P,I)$ with
$(X_{\lambda},Y_{\lambda})\to (XP,PY)$ in SOT, then
$(X_{\lambda}+X(I-P),Y_{\lambda}+(I-P)Y)\in \ms(P,I)$ (indeed
$(X_{\lambda}+X(I-P))P(Y_{\lambda}+(I-P)Y) =
X_{\lambda}PY_{\lambda} = 1$). Since
$(X_{\lambda}+X(I-P),Y_{\lambda}+(I-P)Y) \to (X,Y)$  we get that
$(X,Y)\in \overline{\ms(P,I)}^s$.
\end{proof}

The proof of the next lemma is immediate.

\begin{lemma}\label{cond}
{\rm (i)} If $(X,Y)\in \ms(F_1AF_2,I)$, then $(XF_1,F_2Y)\in
\ms(A,I)$.

In particular

{\rm (ii)} If $F_1AF_2\succ  I$, $\ker(F_1) = 0$ and $\overline{F_2\mh} = \mh$ then $A\succ I$.
\end{lemma}

\begin{lemma}\label{dirsum}
Let $A = 0\oplus A_1$, where $A_1$ acts on infinite-dimensional
space and is invertible. Then $A\succ I$.
\end{lemma}
\begin{proof}
Let $F = I\oplus A_1^{-1}$ then $F$ is invertible and $FA$ is a projection of infinite rank. Hence $FA\succ 1$,
by Lemma \ref{proj}. Using Lemma \ref{cond} (ii), we get that $A\succ 1$.
\end{proof}

\begin{lemma}\label{posit}
If $A\ge 0$ and $A$ is not compact, then $A\succ I$.
\end{lemma}
\begin{proof}
For each ${\varepsilon}>0$, let $P_{\varepsilon} = I - Q$, where $Q$ is the spectral projection of $A$
corresponding to the interval $(0,{\varepsilon})$. Then $P_{\varepsilon}A$ is of the form $0\oplus B$, where $B$
is invertible and, for sufficiently small ${\varepsilon}$, non-compact. Hence $P_{\varepsilon}A\succ I$. By
Lemma \ref{cond}, $\ml(\mh)P_{\varepsilon}\times \ml(\mh) \subset \overline{\ms(A,I)}^s$. Since
$P_{\varepsilon}\to I$ when ${\varepsilon}\to 0$, we get that $A\succ I$.
\end{proof}

\begin{theorem}\label{noncomp}
If $A$ is non-compact, then the set of all solutions of the
equation \eqref{bilin} is \SOT-dense in $\ml(\mh)\times \ml(\mh)$
for each $B$.
\end{theorem}

\begin{proof} It suffices to show that $A\succ I$. Suppose firstly that the operator $U$ in the polar
decomposition $A = UT$ of $A$ is an isometry. The operator $AU^* =
UTU^*$ is non-negative and non-compact. Hence $AU^*\succ  I$.
Since $\overline{U^*\mh}=\mh$, $A\succ I$.

If $U$ is a coisometry, then $U^*A = T$ is a positive non-compact
operator. So $T\succ I$, and since $\ker(U^*)= 0 $, we get $A\succ
I$.
\end{proof}

\section{$A$-expanding operators}\label{S:obst}

In operator theory, especially in dealing with interpolation problems, one often needs to consider Hilbert (or
Banach) spaces with two norms and study operators with special properties with respect to these norms.  The main
purpose of this section is to show that Kolmogorov $n$-widths can be used to describe \WOT-closures of some sets
of operators given by conditions of this kind. Our interest to such conditions is inspired by the theory of
linear fractional relations, see \cite{KOS06} and \cite{KOS07}.
\medskip

Let $\mx$ be a Banach space and $A\in\ml(\mx )$ be a compact operator with an infinite-dimensional range. It
determines a semi-norm $\|x\|_A = \|Ax\|$ on $\mx$. We consider the set $\me(A)$ of all operators $T$ that
increase this semi-norm: $\|Tx\|_A\ge \|x\|_A$ for each $x\in \mx$. In other words

\begin{equation}\label{E:theSet}\me(A):=\{T\in\ml(\mx ):~||ATx||\ge ||Ax||~\forall
x\in \mx \}.\end{equation}

It turns out that the problem of description of $\me(A)$ is a dual
version of the problem considered in previous sections: the
following dual characterization of $\me(A)$ relates it with
covering operators.

\begin{lemma}\label{L:iff} Let a Banach space $\mx$ be reflexive.
An operator $R\in\ml(\mx )$ satisfies $R\in\me(A)$ if and only if $R^*\in G(K)$, where $K=A^*(B_{\mx ^*})$.
\end{lemma}

\begin{proof}
 Assume that $R^*\in G(K)$, that is,
$R^*K\supset K$. Then
\begin{equation}\label{E:ATx_vs_Ax}||ARx||=\sup_{f\in
B_{\mx ^*}}|f(ARx)|=\sup_{f\in B_{\mx ^*}}|(R^*A^*f)(x)|\ge\sup_{f\in B_{\mx ^*}}|(A^*f)(x)|=||Ax||,
\end{equation}
for each $x\in X$. Thus $R\in \me(A)$.
\medskip

Conversely, if $R^*\notin G(K)$, then there is $f\in K\setminus
R^*K$. The set $R^*K$ is weakly closed. By the Hahn--Banach
theorem and reflexivity of $\mx$ there is $x\in \mx$ with $|f(x)|
> \sup_{g\in R^*K}|g(x)| = \|ARx\|$. Since $|f(x)| \le \|Ax\|$ we
obtain that $R\notin \me(A)$.
\end{proof}

We denote the \WOT-closure of $\me(A)$ by $\we(A)$.

\begin{corollary}
Let $\mx$ be a reflexive Banach space, $A$ an operator on $\mx$. Then $\{R^*: R\in \we(A)\}= WG(K)$, where
$K=A^*(B_{\mx ^*})$.
\end{corollary}
\begin{proof}
Since $\mx$ is reflexive the map $R\to R^*$ from $\ml(\mx) $ to
$\ml(\mx^*)$ is bicontinuous in the \WOT-topologies. Hence the
result follows from Lemma \ref{L:iff}.
\end{proof}

\begin{corollary}\label{T:NotEverything} Let $\mx$ be reflexive. If $A\in\ml(\mx )$ is such that the
sequence \[\{d_n(A^*(B_{\mx ^*}))\}_{n=0}^\infty\] is lacunary,
then $\we(A)$ is contained in the set of  all operators for which
$\ker A$ is an invariant subspace.
\end{corollary}

\begin{proof} Follows immediately from Lemma \ref{bystr} if we take into account the observation
that $\ker A$ is an invariant subspace of $R$ if and only if
$\overline{A^*\mx ^*}$ is an invariant subspace of $R^*$ (that is,
if and only if $R^*\in \mathcal{A}_K$).
\end{proof}

Applying Theorem \ref{vklyuch}, we obtain the converse inclusion:
\begin{corollary}\label{T:invariant} The set of all operators preserving $\ker A$ is contained in
$\we(A)$. If $\ker A=\{0\}$, then $\we(A)=\ml(\mx)$.
\end{corollary}

Applying Theorem \ref{T:Hilbert+K}, we get

\begin{corollary}\label{T:Hilbert} If $\mx$ is a separable Hilbert
space and $A\in\ml(\mx)$ is such that the sequence $s$-numbers of
$A$ is not lacunary, then $\we(A)=\ml(\mx)$.
\end{corollary}

We can summarize Hilbert-space-case results in the following way:

\begin{theorem}[A complete classification in the Hilbert space
case]\label{T:completeHilbert}  Let $\mx $ be a se\-parable
Hil\-bert space.
\smallskip

\noindent{\rm (i)} If the sequence of $s$-numbers of $A$ is not
lacunary, then $\we(A)=\ml(\mx)$.
\smallskip

\noindent{\rm (ii)} If the sequence of $s$-numbers of $A$ is
lacunary, then $\we(A)$ coincides with the set of operators for
which $\ker A$ is an invariant subspace.
\end{theorem}

Finally, using Theorem \ref{T:ultra} we obtain a result on the
ultra-weak closure of $\we(A)$:

\begin{corollary} Let $A\in\ml(\mh)$ be such that its sequence of $s$-numbers is not lacunary. Then the
closure of the set {\rm (\ref{E:theSet})} in the ultra-weak
topology coincides with $\ml(\mh)$.
\end{corollary}

\end{large}
\end{document}